\documentclass[11pt]{article}
\usepackage{amssymb}
\usepackage{amsthm}
\usepackage{amsmath}
\usepackage{fullpage,hyperref}
\usepackage{color,graphics}
\usepackage{enumerate}
\usepackage{graphicx} 

\newtheorem{theorem}{Theorem}[section]
\newtheorem{lemma}[theorem]{Lemma}

\newtheorem{cor}[theorem]{Corollary}

\newtheorem{definition}[theorem]{Definition}

\newtheorem{remark}{Remark}[section]
\newtheorem{notation}{Notation}[section]

\newtheorem{question}{Question}

\newcommand{\R}{\ensuremath{\mathbb{R}}}
\newcommand{\E}{\ensuremath{\mathbb{E}}}
\newcommand{\pr}{\ensuremath{\mathbb{P}}}

\newcommand{\secref}[1]{Sec.~\ref{#1}}
\newcommand{\thmref}[1]{Thm.~\ref{#1}}
\newcommand{\lemref}[1]{Lemma~\ref{#1}}
\newcommand{\corref}[1]{Cor.~\ref{#1}}

\newcommand{\defref}[1]{Definition~\ref{#1}}

\newcommand{\Phibar}{\overline{\Phi}}

\title{Finding Hidden Cliques in Linear Time with High Probability}

\author{Yael Dekel \\ The Hebrew University \\ {\tt yaelvin@cs.huji.ac.il} 
\and Ori Gurel-Gurevich \\ University of British Columbia \\ {\tt origurel@math.ubc.ca}
\and Yuval Peres \\ Microsoft Research \\ {\tt peres@microsoft.com}}

\begin{document}

\date{}
\maketitle

\begin{abstract}
We are given a graph $G$ with $n$ vertices, where a random subset of
$k$ vertices has been made into a clique, and the remaining edges are
chosen independently with probability $\tfrac12$. This random graph model
is denoted $G(n,\tfrac12,k)$. The hidden clique problem is to design an
algorithm that finds the $k$-clique in polynomial time with high
probability. An algorithm due to Alon, Krivelevich and Sudakov uses
spectral techniques to find the hidden clique with high probability
when $k = c \sqrt{n}$ for a sufficiently large constant $c >
0$. Recently, an algorithm that solves the same problem was proposed
by Feige and Ron. It has the advantages of being simpler and more
intuitive, and of an improved running time of $O(n^2)$. However, the
analysis in the paper gives success probability of only $2/3$. In this
paper we present a new algorithm for finding hidden cliques that both
runs in time $O(n^2)$, and has a failure probability that is less than
polynomially small.
\end{abstract}

\section{Introduction}
A clique in a graph $G$ is a subset of its vertices any two of which
are connected by an edge. The problem of determining the size of the
maximum clique in a graph is known to be NP-complete \cite{Karp72}. It
has also been proved \cite{Feigeetal91,AroraSafra92,Aroraetal92} that
assuming P $\neq$ NP, there exists a constant $b>0$ for which it is
hard to approximate the size of the maximum clique within a factor of
$n^b$. Therefore, it is natural to investigate the hardness of this
problem in the average case. 

The Erd\"{o}s R\'{e}nyi random graph model, also denoted $G(n,\tfrac12)$,
is a probability measure on graphs with $n$ vertices. In this model, a
random graph is generated by choosing each pair of vertices
independently with probability $\tfrac12$ to be an edge. It is known that
with probability tending to $1$ as $n$ tends to infinity, the size of
the largest clique in $G(n,\tfrac12)$ is $(2 + o(1)) \log n$. There exists
a polynomial time algorithm (see for example
\cite{GrimmetMcDiarmid75}) that finds a clique of size $(1 + o(1))
\log n$ in $G(n,\tfrac12)$ with high probability, but even though in
expectation $G(n,\tfrac12)$ contains many cliques of size $(1 +
\varepsilon) \log n$ for any fixed $0 < \varepsilon < 1$, there is no
known polynomial time algorithm that finds one. It is plausible to
conjecture that this problem is computationally hard, and this
hardness has been used in several cryptographic applications
\cite{Kucera91,JuelsPeinado00}.

Finding a large clique may be easier in models where the graphs
contain larger cliques. Define, therefore, the hidden clique model,
denoted by $G(n,\tfrac12,k)$. In this model, a random $n$ vertex graph
is generated by randomly choosing $k$ vertices to form a clique, and
choosing every other pair of vertices independently with probability
$\tfrac12$ to be an edge. Jerrum \cite{Jerrum92} and Ku\v{c}era
\cite{Kucera95} suggested this model independently and posed the
problem of finding the hidden clique. When $k \geq c_0 \sqrt{n \log
  n}$ for some sufficiently large constant $c_0$, Ku\v{c}era observed
\cite[\S Thm. 6.1]{Kucera95} that the hidden clique can be found with
high probability by taking the $k$ highest degree vertices in the
graph. For $k = c \sqrt{n}$, there is an algorithm due to Alon,
Krivelevich and Sudakov \cite{AlonKrivelevichSudakov98} that finds the
hidden clique with high probability when $c$ is sufficiently large
using spectral techniques. In a more recent paper \cite{FeigeRon10},
Feige and Ron propose a simple algorithm that runs in time $O(n^2)$
and finds the hidden clique for $k=c \sqrt{n}$ with probability at
least $2/3$. In this paper we present a new algorithm that has the
advantages of both algorithms, as it runs in time $O(n^2)$, and fails
with probability that is less than polynomially small in $n$. The
algorithm has three phases, and it uses two parameters: $0 < \alpha <
1$ and $\beta > 0$. In the first phase, we iteratively find subgraphs
of the input graph $G$. Denote these subgraphs by $G = G_o \supset G_1
\supset G_2 \supset \cdots$. Given $G_i$, we define $G_{i+1}$ as
follows: Pick a random subset of vertices $S_i \subseteq V(G_i)$ that
contains each vertex with probability $\alpha$. Define $\tilde{V}_i$
as the set that contains all the vertices in $G_i$ that are not in
$S_i$, that have at least $\tfrac12 |S_i| + \beta
\tfrac{\sqrt{|S_i|}}2$ neighbors in $S_i$, namely
\[
\tilde{V}_i = \big\{ v \in V(G_i) \setminus S_i: \big| \big\{ u \in
S_i: (u,v) \in E(G_i) \big\} \big| \geq \tfrac12 |S_i| + \beta
\tfrac{\sqrt{|S_i|}}2 \big\}~.
\]
Define $G_{i+1}$ to be the induced subgraph of $G_i$ containing only
the vertices in $\tilde{V}_i$. We choose $\alpha$ and $\beta$ in such
a way that the relative size of the hidden clique grows with each
iteration.  We repeat the process $t$ times, until we are left with a
subgraph where the hidden clique is large enough so we can continue to
the second phase. A logarithmic number of iterations is enough. For
the exact way of choosing $\alpha, \beta$ and $t$, see the proof of
\lemref{totalfailureprobability}.

In the second phase, we find $\tilde{K}$, the subset of the hidden
clique contained in $G_t$. This is done by estimating the number of
clique vertices in $G_t$ by $k_t$, then defining $K'$ as the set of
$k_t$ largest degree vertices in $G_t$, and letting $\tilde{K}$
contain all the vertices in $G_t$ that have at least $\tfrac{3k_t}4$
neighbors in $K'$. In the third phase of the algorithm, we find the
rest of the hidden clique using $\tilde{K}$. This is done by letting
$G'$ be the induced subgraph of $G$ containing $\tilde{K}$ and all its
common neighbors. Let $K^*$ be the set of the $k$ largest degree
vertices in $G'$. Then $K^*$ is the set returned by the algorithm as
the candidate for the hidden clique.

\begin{theorem}
\label{thetheorem}
If $c > c_0$ then there exist $\alpha, \beta$ such that, given $G \in
G(n,\tfrac12, c \sqrt{n})$, the probability that $K^* = K^*(\alpha,
\beta)$ is the hidden clique is at least $1 -
\mathrm{e}^{-\Theta(n^{\varepsilon_0})}$ for some $\varepsilon_0 =
\varepsilon_0(c)$.
\end{theorem}
Numerical calculations show that $c_0$ is close to $1.65$. For a
mathematical definition of $c_0$ see \defref{c0}. A refinement of the
algorithm that works with high probability for all $c \geq 1.261$ is
presented in \secref{variation}.

\subsection{Related Work}
Since \cite{AlonKrivelevichSudakov98}, there have been many papers
describing algorithms that solve various variants of the hidden clique
problem. In \cite{FeigeKrauthgamer00} an algorithm for finding hidden
cliques of size $\Omega(\sqrt{n})$ based on the Lov\'{a}sz theta
function is given, that has two advantages. The first is being able to
find the clique also in a semi-random hidden clique model, in which an
adversary can remove edges that are not in the clique, and the second
is being able to certify the optimality of its solution by providing
an upper bound on the size of the maximum clique in the graph.

McSherry \cite{McSherry01} gives an algorithm that solves the more
general problem of finding a planted partition. In the random graph
model described there, we are given a graph where the vertices are
randomly partitioned into $m$ classes, and between every pair of
vertices where one is in class $i$ and the other in class $j$ there is
an edge with probability $p_{ij}$. With the appropriate parameters,
this model can be reduced both to the hidden clique model and to the
hidden dense graph model that we describe in
\secref{hiddendensegraph}. For both these cases, the result is a
polynomial time algorithm that finds the hidden clique (dense graph)
with high probability for $k=c \sqrt{n}$.

Several attempts have been made to develop polynomial time algorithms
for finding hidden cliques of size $k = o(\sqrt{n})$, so far with no
success. For example, Jerrum \cite{Jerrum92} described the Metropolis
process and proved that it cannot find the clique when $k =
o(\sqrt{n})$. Feige and Krauthgamer \cite{FeigeKrauthgamer03} explain
why the algorithm described in \cite{FeigeKrauthgamer00} fails when $k
= o(\sqrt{n})$. Frieze and Kannan \cite{FriezeKannan08} give an
algorithm to find a hidden clique of size $k = \Omega \left( n^{1/3}
\log^4 n \right)$, however, the algorithm maximizes a certain cubic
form, and there are no known polynomial time algorithms for maximizing
cubic forms. In \secref{oracle} we give an algorithm that finds the
hidden clique when we are given a small part of it by an oracle or an
adversary. We prove, that for any $k = \omega (\log n \log \log n)$,
knowing only $\log n + 1$ vertices of the hidden clique enables us to
find the rest of them with high probability. For smaller $k$'s, $\log
n + 1$ is not enough, but $(1 + \varepsilon) \log n$ is.

There are many problems in different fields of computer science that
are related to the hidden clique problem. Among others, there are
connections to cryptography, testing and game theory. For connections
to cryptography, see for example \cite{Kucera91} where an encryption
scheme based on hiding an independent set in a graph is described or
\cite{JuelsPeinado00} where the function whose input is a graph $G$
and a set $K$ of $k$ vertices and whose output is $G$ with a clique on
$K$ is proposed as a one way function for certain values of $k$. For
connections to testing, see \cite{AlonAKMRX07} where Alon et al. prove
that if there is no polynomial time algorithm to find hidden cliques
of size $t > \log^3 n$ then there is no polynomial time algorithm that
can test $k$-wise independence of a distribution even when given a
polynomial number of samples from it, for $k = \Theta (\log n)$. For
connections to game theory, see \cite{HazanKrauthgamer09}, where Hazan
and Krauthgamer prove that if there is a polynomial time algorithm
that finds a Nash equilibrium of a two player game whose
social-welfare is close to the maximum, then there is a randomized
polynomial time algorithm that finds the hidden clique for $k = O(\log
n)$. The hidden clique model is also related to the planted-SAT model
\cite{BenSassonBiluGutfreund02,KrivelevichVilenchik06} and some models
in computational biology \cite{BenDorShamirYakhini99}.

\section{Proof of \thmref{thetheorem}}
Throughout the paper we use the following notations.
\begin{notation}
Given a graph $G = (V,E)$, for every $v \in V$ and $S \subseteq V$ we
denote by $d_S(v)$ the number of neighbors $v$ has in $S$. Formally,
\[
d_S(v) = \left| \left\{ u \in S : (u,v) \in E \right\} \right|~.
\]
We abbreviate $d_V(v)$ by $d(v)$.
\end{notation}
\begin{notation}
Let $\varphi(x)$ denote the Gaussian probability density function
$\varphi(x) = \frac1{\sqrt{2 \pi}} \mathrm{e}^{-x^2/2}$. We denote by
$\Phi(x)$ the Gaussian cumulative distribution function $\Phi(x) =
\int_{-\infty}^x \varphi(t) dt$, and $\Phibar (x) = 1 - \Phi(x)$.
\end{notation}
\begin{notation}
All logarithms in the paper are base $2$.
\end{notation}
\begin{notation}
We use the shorthand ``whp($f(n)$)'' to mean: ``with probability at
least $1 - f(n)$''.
\end{notation}
\begin{definition}
Given $0 < \alpha < 1$ and $\beta > 0$, we define
\[
\tau = (1 - \alpha) \Phibar(\beta)
\]
and
\[
\rho = (1 - \alpha) \Phibar( \beta - c \sqrt{\alpha})~.
\]
\end{definition}
\begin{definition}
\label{c0}
For every $\alpha, \beta$, denote the minimal $c$ for which $\rho \ge
\sqrt{\tau}$ by $\tilde{c}(\alpha, \beta)$. Define $c_0$ as the
infimum of $\tilde{c}(\alpha, \beta)$ for $0 < \alpha < 1$ and $\beta
> 0$.
\end{definition}
\begin{definition}
Define $n = n_0, n_1, n_2, \ldots$ and $k = k_0, k_1, k_2, \ldots$ by
$n_i = \tau^i n$ and $k_i = \rho^i k$. Define also $n = \tilde{n}_0,
\tilde{n}_1, \ldots$ and $k = \tilde{k}_0, \tilde{k}_1,\ldots$ to be
the actual sizes of $G_i$ and the hidden clique in $G_i$ respectively
when running the algorithm.
\end{definition}

\subsection{Proving the correctness of the algorithm}
In order to prove the correctness of the algorithm, we examine each of
the three phases of the algorithm. First, we prove that in every
iteration, with high probability $\tilde{n}_i, \tilde{k}_i$ are close
to $n_i, k_i$ respectively. We do this by first proving that in every
iteration the graph $G_i$ is a copy of $G(\tilde{n}_i, \tfrac12,
\tilde{k}_i)$, and therefore it is enough to prove that given a graph
in $G(n, \tfrac12, k)$, with high probability $\big| \tilde{V}_0
\big|$ is close to $\tau n$ and $\big| \tilde{V}_0 \cap K \big|$ is
close to $\rho k$. Here, the high probability should be high enough to
remain high even after $t$ iterations. Next, we prove that with high
probability $\tilde{K}$ is a subset of the hidden clique. Last, we
prove that with high probability $K^*$ is the hidden clique.

\subsubsection{Proving the correctness of the first phase of the algorithm}
\begin{lemma}
\label{nextiterationisrandom}
For every $i \geq 0$, the graph $G_i$ defined in the $i$'th iteration
of the algorithm is a copy of $G(\tilde{n}_i, \tfrac12, \tilde{k}_i)$.
\end{lemma}
\begin{proof}
We prove this by induction. Assume that $G_i$ is a copy of
$G(\tilde{n}_i,\tfrac12,\tilde{k}_i)$. Consider the following
equivalent way of generating $G(\tilde{n}_i,\tfrac12,\tilde{k}_i)$:
First, pick the $\tilde{k}_i$ hidden clique vertices. Then pick the
set $S_i$. Then pick all the edges between $V(G_i) \setminus S_i$ and
$S_i$. At this point, we still need to pick the edges in $S_i$ and in
$V(G_i) \setminus S_i$, but we already have enough information to find
$\tilde{V}_i$, which is the vertex set of $G_{i+1}$. Since we can find
the vertices of $G_{i+1}$ before exposing any of the edges in it, it
is a copy of $G(\tilde{n}_{i+1},\tfrac12,\tilde{k}_{i+1})$.
\end{proof}

\begin{lemma}
\label{sizeofS}
For every $0 < \varepsilon_1 < \tfrac12$ and $0 < \varepsilon_2 <
\tfrac12$, the set $S_0$ satisfies $\big| |S_0| - \alpha n \big| \leq
O(n^{1 - \varepsilon_1})$ and $\big| |S_0 \cap K| - \alpha k \big|
\leq O(k^{1 - \varepsilon_2})$ whp($\mathrm{e}^{- \Theta (n^{1 - 2
    \varepsilon_1})} + \mathrm{e}^{-\Theta(k^{1 - 2
    \varepsilon_2})}$).
\end{lemma}
\begin{proof}
Follows directly from \thmref{hoeffding}, by setting $t = n^{1 -
  \varepsilon_1}$ for the bound on $|S_0|$ and $t = k^{1 -
  \varepsilon_2}$ for the bound on $|S_0 \cap K|$.
\end{proof}

\begin{lemma}
\label{sizeofnextiteration}
For every $0 < \varepsilon_1 < \tfrac12$ and $0 < \varepsilon_2 <
\tfrac12$, the set $\tilde{V}_0$ satisfies $ \big| | \tilde{V}_0 | -
\tau n \big| \leq O(n^{1 - \varepsilon_1})$ and $ \left| | \tilde{V}
\cap K | - \rho k \right| \leq O(k^{1 - \varepsilon_2})$
whp($\mathrm{e}^{- \Theta (n^{1 - 2 \varepsilon_1})} +
\mathrm{e}^{-\Theta(k^{1 - 2 \varepsilon_2})}$).
\end{lemma}
\begin{proof}
Assume that the events $|S_0| = (1+o(1))\alpha n$ and $|S_0 \cap K| =
(1 + o(1)) \alpha k$ both occur. By \lemref{sizeofS} this happens with
high probability. We can now apply
\corref{numberabovethresholdindependent} twice.

For the vertices in $\left( V \setminus S_0 \right) \setminus K$, the
result follows directly from \corref{numberabovethresholdindependent}
by setting $\varepsilon = \varepsilon_1$.  For $v \in \left( V
\setminus S_0 \right) \cap K$, having $d_{S_0}(v) \geq \tfrac12 \alpha
n + \beta \tfrac{\sqrt{\alpha n}}2$ is equivalent to having
\[
d_{S_0
  \setminus K}(v) \geq \tfrac12 \alpha (n - k) + \tfrac12 \big( \beta
- c \sqrt{\alpha} \big)\sqrt{\tfrac{n}{n-k}} \sqrt{\alpha (n -
  k)}~.
\]
So setting $\varepsilon = \varepsilon_2$ in
\corref{numberabovethresholdindependent}, gives that
\[
\pr \big( \big| |\tilde{V}_0 \cap K| - \rho'k \big| \leq O(k^{1 -
  \varepsilon_2}) \big) \geq 1 - \mathrm{e}^{-\Theta(k^{1 - 2
    \varepsilon_2})}~,
\]
where $\rho' = (1-\alpha) \Phibar \big( (\beta - c \sqrt{\alpha})
\sqrt{\tfrac{n}{n-k}} \big)$. But the difference between $\rho$ and
$\rho'$ is of order $\tfrac1{\sqrt{n}}$, which means that the result
holds for $\big| | \tilde{V}_0 \cap K | - \rho k \big|$ as well.
\end{proof}

\begin{remark}
In order to get a success probability that tends to $1$, we need to
bound the sum of the probabilities of failing in each iteration by
$o(1)$. We refer the reader to \secref{failureprobanalysis} for a
detailed analysis of the failure probability of the algorithm.
\end{remark}

\subsubsection{Proving the correctness of the second phase of the algorithm}
We start by bounding the probability that a hidden clique of size $k$
contains the $k$ largest degree vertices in the graph.
\begin{lemma}
\label{probcliqueverticesaremaximal}
Let $G \in G(n,\tfrac12,k)$. Then whp($\mathrm{e}^{-(k^2 / 8 n - \log
  n - O(1))}$), the clique vertices are the $k$ largest degree
vertices in the graph. Formally, if we denote the hidden clique by
$K$, and the set of $k$ largest degree vertices by $M$, then
\[
\pr \big( \big| K \setminus M \big| > 0 \big) \leq
\mathrm{e}^{-(k^2/8n - \log n - O(1))}~.
\]
\end{lemma}
\begin{proof}
Define $x = \tfrac14 k$. Then by \thmref{hoeffding}
\[
\pr \big( \exists v \not\in K : d(v) \geq \tfrac12 n + x \big) \leq n
\pr \big( B \big( n,\tfrac12 \big) \geq \tfrac12 n + x \big) \leq n
\pr \big( \big| B \big(n, \tfrac12 \big) - \tfrac12 n \big| \geq x
\big) \leq 2 n \mathrm{e}^{-k^2 / 8n}~.
\]
On the other hand,
\begin{eqnarray*}
\pr \big( \exists v \in K : d(v) < \tfrac12 n + x \big) & \leq & k \pr
\big( B \big( n - k,\tfrac12 \big) < \tfrac12 (n-k) + x - \tfrac12 k
\big) \\ & \leq & k \pr \big( \big| B \big(n-k, \tfrac12 \big) -
\tfrac12 (n-k) \big| \geq x \big) \leq 2 k \mathrm{e}^{-k^2 / 8n}~.
\end{eqnarray*}
Therefore, the probability that there exist a non-clique vertex $v$
and a clique vertex $u$ such that $d(u) < d(v)$ is bounded by $2 (n+k)
\mathrm{e}^{-k^2/8n}$.
\end{proof}

\begin{cor}
\label{failureprobabilitytildeK}
If the algorithm does $t$ iterations before finding $\tilde{K}$ and
succeeds in every iteration, then whp($\mathrm{e}^{-\Theta ( (
  \frac{\rho^2}{\tau} )^t )}$), $\tilde{K}$ is a subset of the
original hidden clique.
\end{cor}
\begin{proof}
The algorithm estimates $\tilde{k}_t$, the number of hidden clique
vertices in $G_t$, by $k_t = \rho^t k$. If the input graph has $n$
vertices and a hidden clique of size $k = c \sqrt{n}$, and all the
iterations are successful, then $|\tilde{k}_t - k_t| \leq O(k_t^{1 -
  \varepsilon_1})$. Recall that $K'$ is defined as the $k_t$ largest
degree vertices in $G_t$. By \lemref{probcliqueverticesaremaximal},
whp($\mathrm{e}^{-\Theta ( \frac{\rho^{2t}k^2}{\tau^t n} )}$) the
hidden clique vertices have the largest degrees in $G_t$, so if
$\tilde{k}_t < k_t$ then $K'$ contains all the hidden clique vertices
in $G_t$ plus $O(k_t^{1 - \varepsilon_1})$ non-clique vertices, and if
$\tilde{k}_t > k_t$, then $K'$ contains all the hidden clique vertices
in $G_t$ except for $O(k_t^{1 - \varepsilon_2})$ of them. In both
cases, every hidden clique vertex in $G_t$ has at least $k_t -
O(k_t^{1 - \varepsilon_2})$ neighbors in
$K'$. Whp($\mathrm{e}^{-\Theta ( \frac{\rho^{2t}k^2}{\tau^t n} )}$)
every non-clique vertex in $G_t$ has at most $\tfrac{2k_t}3$ neighbors
in $K'$ (this follows from \thmref{hoeffding} and the union
bound). Thus, if we define $\tilde{K} = \big\{ v \in V(G_t) :
d_{K'}(v) \geq \tfrac{3k_t}4 \big\}$, then whp($\mathrm{e}^{-\Theta (
  \frac{\rho^{2t}k^2}{\tau^t n} )}$), $\tilde{K}$ contains every
clique vertex in $G_t$, and no non-clique vertex in $G_t$.
\end{proof}

\subsubsection{Proving the correctness of the third phase of the algorithm}
In order to prove that $K^*$ is the hidden clique with high
probability, we prove a more general Lemma. We prove that if an
adversary reveals a subset of the clique that is not too small, we can
use it to find the whole clique.
\label{oracle}
\begin{lemma}[Finding hidden cliques from partial information]
\label{findKgivensubset}
We are given a random graph $G \in G(n,\tfrac12,k)$, and a subset of
the hidden clique $\tilde{K} \subseteq K$ of size $s$. Suppose that
either
\begin{enumerate}[(a)]
\item $k = O(\log n \log \log n)$ and $s \geq (1 + \varepsilon) \log
  n$ for some $\varepsilon > 0$, or
\item $k \geq \omega (\log n \log \log n)$ and $s \geq \log n + 1$.
\end{enumerate}
Let $G'$ denote the subgraph of $G$ induced by $\tilde{K}$ and all its
common neighbors, and define $K^*$ to be the $k$ largest degree
vertices of $G'$. Then for every $0 < \varepsilon_3 < \tfrac12$,
whp($\mathrm{e}^{-\Theta(s \log k + \log n)} +
\mathrm{e}^{-\Theta(k^{1 - 2 \varepsilon_3})}$), $K^* = K$.
\end{lemma}
\begin{proof}


Consider an arbitrary subset of $K$ of size $s$. The probability that
its vertices have at least $l_0$ non-clique common neighbors can be
bounded by $\sum_{l=l_0}^{n-k} n^l 2^{-sl}$. Taking union bound over
all subsets of size $s$ of $K$ gives that the probability that there
exists a subset with at least $l_0$ non-clique common neighbors is
bounded by
\[
k^s \sum_{l=l_0}^{n-k} n^l 2^{-sl} = \sum_{l=l_0}^{n-k} 2^{s \log k +
  l (\log n - s)} \leq n 2^{s \log k + l_0 (\log n - s)}~.
\]
Therefore, this is also a bound on the probability that the set
$\tilde{K}$ has at least $l_0$ non-clique neighbors. So we have
\[
\pr \big( \big| V \big( G' \big) \big| \geq k + l_0 \big) \leq 2^{\log
  n + s \log k + l_0 (\log n - s)}~.
\]
By our assumptions on $s$, we know that $\log n - s$ is
negative. Therefore, we can take $l_0 = \tfrac{2 ( \log n + s \log
  k)}{s - \log n}$ and get that whp($2^{-s \log k - \log n}$), there
are at most $l_0$ non-clique vertices that are adjacent to all of
$\tilde{K}$. Recall that the probability that there exists a
non-clique vertex in $G$ with more than $\tfrac{k}2 + k^{1 -
  \varepsilon_3}$ neighbors in the hidden clique is bounded by
$\mathrm{e}^{-\Theta(k^{1 - 2 \varepsilon_3})}$. Therefore,
whp($\mathrm{e}^{-\Theta(k^{1 - 2 \varepsilon_3})}$), the degrees of
all the non-clique vertices in $G'$ are at most $\tfrac{k}2 + k^{1 -
  \varepsilon_3} + l_0$. If $s$ and $k$ are such that $l_0 = o(k)$,
this value is smaller than $k-1$. On the other hand, all the clique
vertices in $G'$ have degree at least $k - 1$, so the clique vertices
have the largest degrees in $G'$.

If $k = \omega ( \log n \log \log n)$ then letting $s = \log n + 1$
gives $l_0 = 2 \big( \log n + \log n \log k + \log k \big)$.  Clearly,
$\log n + \log k = o(k)$. To see that $\log n \log k = o(k)$, denote
$k = \log n f(n)$ where $f(n) = \omega (\log \log n)$. Then $\log n
\log k = \log n \big( \log \log n + \log \big( f(n) \big)
\big)$. Clearly, $\log n \log (f(n)) = o(\log n f(n))$, and from the
definition of $f(n)$ we also have $\log n \log \log n = o(\log n
f(n))$.

If $k \leq O(\log n \log \log n)$, then letting $s \geq (1 +
\varepsilon) \log n$ for some small $\varepsilon > 0$ is enough, since
then $l_0 = \tfrac2{\varepsilon} +
\tfrac{2(1+\varepsilon)}{\varepsilon} \log k = o(k)$.
\end{proof}

\subsection{Bounding the failure probability}
\label{failureprobanalysis}

\begin{lemma}
\label{totalfailureprobability}
For every $c > c_0$, there exist $0 < \alpha < 1$ and $\beta > 0$ such
that if we define $a = - \tfrac{\log \tau}{\log \frac{\rho^2}{\tau}}$
and $b = - \tfrac{\log \rho^2}{\log \frac{\rho^2}{\tau}}$, then for
every $\varepsilon_0 < \tfrac1{a}$, the failure probability of the
algorithm is at most $\mathrm{e}^{-\Theta(n^{\varepsilon_0})}$.
\end{lemma}
\begin{proof}
In order for the probability proven in
\corref{failureprobabilitytildeK} to tend to $0$, we need $\tau$ and
$\rho$ to satisfy $\tfrac{\rho}{\sqrt{\tau}} > 1$. From \defref{c0} we
know that for $c > c_0$ there exist $\alpha, \beta$ that satisfy this
inequality. Numerical calculations show that $c_0$ is close to
$1.65$. The values of $\alpha$ and $\beta$ for which
$\tilde{c}(\alpha, \beta) = 1.65$ are $\alpha = 0.3728$ and $\beta =
0.72$. For these values, we get $\tau \approx 0.14787$ and $\rho
\approx 0.38455$, and $\tfrac{\rho}{\sqrt{\tau}} \approx 1.00003$.

Let the number of iterations be $t = \tfrac{ \varepsilon_4 \log
  n}{\log \frac{\rho^2}{\tau}}$ for some $0 < \varepsilon_4 <
\tfrac1{a}$. We use the union bound to estimate the failure
probability during the iteration phase of the algorithm. By Lemmas
\ref{sizeofS} and \ref{sizeofnextiteration}, this probability is at
most $\sum_{i=0}^t \big( \mathrm{e}^{-\Theta(n_i^{1 - 2
    \varepsilon_1})} + \mathrm{e}^{-\Theta(k_i^{1 - 2 \varepsilon_2})}
\big)$, which can be upper bounded by
\[
\mathrm{e}^{-\Theta(n^{(1 - 2
    \varepsilon_1)(1 - \varepsilon_4 a)})} +
\mathrm{e}^{-\Theta(n^{\frac12 (1 - 2 \varepsilon_2) (1 -
    \varepsilon_4 b)})}~.
\]
By \corref{failureprobabilitytildeK}, the failure probability in the
step of finding $\tilde{K}$ is bounded by
$\mathrm{e}^{-\Theta(n^{\varepsilon_4})}$. Finally, if $t$ is as
defined above, then assuming the first two phases succeed,
$|\tilde{K}| \geq \rho^t k - o( \rho^t k) = k^{1 - b \varepsilon_4} (
1 - o(1))$ (notice that $b = a - 1$ so $\varepsilon_4 < \tfrac1{a}$
implies that $1 - b \varepsilon_4 > 0$). $\tilde{K}$ is large enough
so that we can use \lemref{findKgivensubset}, to conclude that the
probability of failing in the third phase is at most
\[
\mathrm{e}^{-\Theta(n^{\frac12 (1 - \varepsilon_4
    b)} \log n)} + \mathrm{e}^{-\Theta(k^{1 - 2 \varepsilon_3})}~.
\]
For any choice of $0 < \varepsilon_1, \varepsilon_2 < \tfrac12$ and $0
< \varepsilon_4 < \tfrac1{a}$, denote 
\[
\varepsilon_0 = \min \big\{
\varepsilon_4, (1 - 2 \varepsilon_1) (1 - \varepsilon_4 a), \tfrac12
(1 - 2 \varepsilon_2)(1 - \varepsilon_4 b) \big\}~,
\]
and take $\varepsilon_3 = \tfrac{1 - 2 \varepsilon_0}2$ (notice that
$\varepsilon_3 > 0$ because $\varepsilon_0 < \tfrac12$). With these
parameters, the failure probability of the whole algorithm is bounded
by $\mathrm{e}^{-\Theta (n^{\varepsilon_0})}$.
\end{proof}

\section{Refinements}

\subsection{A variation of this algorithm that works for smaller cliques}
\label{variation}
The reason our algorithm works is that the clique vertices in $V(G_i)
\setminus S_i$ have a boost of around $\tfrac12 \alpha k_i$ (which is
$c \sqrt{\alpha}$ times the standard deviation) to their degrees, so
this increases the probability that their degree is above the
threshold. If we could increase the boost of the clique vertices'
degrees (in terms of number of standard deviations) while still
keeping the graph for the next iteration random, then we would be able
to find the hidden clique for smaller values of $c$. One way to
achieve this, is by finding a subset $\tilde{S}_i$ of $S_i$, that has
$\gamma n_i$ vertices ($\gamma < \alpha$) and $\delta k_i$ clique
vertices. If we count just the number of neighbors the vertices in
$V(G_i) \setminus S_i$ have in $\tilde{S}_i$, then the clique vertices
have a boost of around $\tfrac12 \delta k_i$ to their degree, which is
$c \tfrac{\delta}{\sqrt{\gamma}}$ times the standard deviation.

The subset of $S_i$ that we use in this variation is the set of all
vertices $v \in S_i$ that have $d_{S_i}(v) \geq \tfrac12 | S_i | +
\eta \tfrac{\sqrt{| S_i|}}2$, for some $\eta > 0$. Since these degrees
are not independent we cannot use the same concentration results we
used before, so we first prove the following concentration result.

\begin{lemma}
\label{numberabovethresholddependent}
Let $G \in G\big(n,\tfrac12 \big)$ and $a, c' > 0$. Define a random
variable
\[
X = \big| \big\{ v \in V(G) : d(v) \geq \tfrac12 n + a
\tfrac{\sqrt{n}}2 \big\} \big|~.
\]
Then for every $0 < \varepsilon' < \tfrac14$ it holds that
\[
\pr \big( \big| X - \Phibar(a) n \big| \geq c' n^{1 - \varepsilon'}
\big) \leq 2 \mathrm{e}^{-\pi c'^4 n^{1 - 4 \varepsilon'} / 32}~.
\]
\end{lemma}
\begin{proof}
For every $v \in V(G)$ define a random variable
\[
X_v = \left\{ \begin{array}{ll} 1 & d(v) \geq \tfrac12 n + a
  \tfrac{\sqrt{n}}2 \\ 0 & \textrm{otherwise}
\end{array} \right.~.
\]
Then $X = \sum X_v$. By \corref{expectednumberabovethreshold} we have
$| \Phibar(a)n - \E X | \leq c \sqrt{n}$ for some constant $c$.  To
prove that $X$ is concentrated around its mean we define additional
random variables. Let $\varepsilon > 0$ to be defined later, and
define three thresholds:
\[
t_1 = \tfrac12 n + (a - \varepsilon) \tfrac{\sqrt{n}}2
~,~~~~t_2 = \tfrac12 n + a \tfrac{\sqrt{n}}2~,
~~~~\text{and}~~~~
t_3 = \tfrac12 n + (a + \varepsilon) \tfrac{\sqrt{n}}2~.
\]
For every $v \in V(G)$ define
\[
\begin{array}{ll}
F_v = \left\{ \begin{array}{ll} 0 & d(v) < t_1 \\ \frac{2 \big( d(v) -
    t_1 \big)}{\varepsilon \sqrt{n}} & t_1 \leq d(v) \leq t_2 \\ 1 &
  d(v) > t_2
\end{array} \right.~, & G_v = \left\{ \begin{array}{ll} 0 & d(v) < t_2 \\ \frac{2 \big( d(v) - t_2 \big)}{\varepsilon
    \sqrt{n}} & t_2 \leq d(v) \leq t_3 \\ 1 & d(v) > t_3
\end{array} \right.
\end{array}
\]
Define $F = \sum_v F_v$ and $G = \sum_v G_v$. For every $v \in V$, we
bound $\E F_v - \E X_v$ and $\E X_v - \E G_v$.
\begin{eqnarray}
\label{F}
\E F_v - \E X_v & = & 2^{-n} \sum_{i=t_1}^{t_2} \tfrac{2(i -
  t_1)}{\varepsilon \sqrt{n}} \tbinom{n}{i} \leq 2^{-n}
\sum_{i=t_1}^{t_2} \tbinom{n}{i} \leq \tfrac{\varepsilon \sqrt{n}}2
2^{-n} \tbinom{n}{\frac{n}2} \leq \tfrac{\varepsilon}{\sqrt{2 \pi}}
\big( 1 + O \big( \tfrac1{n} \big) \big)
\end{eqnarray}
where the last two inequalities follow from the fact that
$\tbinom{n}{\frac{n}2}$ is the maximal binomial coefficient, and from
Stirling's approximation (see, for example \cite{AS64}): $n! = \sqrt{2
  \pi n} \big( \tfrac{n}{\mathrm{e}} \big)^n \big( 1 + O \big(
\tfrac1{n} \big) \big)$. Repeating this calculation for $\E X_v - \E
G_v$ gives
\begin{equation}
\label{G}
\E X_v - \E G_v = 2^{-n} \sum_{i = t_2}^{t_3} \big( 1 - \tfrac{2(i -
  t_2)}{\varepsilon \sqrt{n}} \big) \tbinom{n}{i} \leq
\tfrac{\varepsilon}{\sqrt{2 \pi}} \big( 1 + O \big(\tfrac1{n} \big)
\big)~.
\end{equation}
From \eqref{F} and \eqref{G} we have 
\[
\pr \big( \big| X - \E X \big| \geq \lambda n \big) \leq \pr \big(F -
\E F \geq \big( \lambda - \tfrac{\varepsilon}{\sqrt{2 \pi}} \big) n
\big) + \pr \big(G - \E G \leq -\big( \lambda -
\tfrac{\varepsilon}{\sqrt{2 \pi}} \big) n \big)~.
\]
Thus, we need to calculate the concentration of $F$ and $G$. Both are
edge exposure martingales with Lipschitz constant $\tfrac2{\varepsilon
  \sqrt{n}}$. Therefore, by Azuma's inequality (see, for example
\cite{McDiarmid89}) we get:
\[
\pr \big( F - \E F \geq \big( \lambda - \tfrac{\varepsilon}{\sqrt{2
    \pi}} \big) n \big) + \pr \big( G - \E G \leq - \big( \lambda -
\tfrac{\varepsilon}{\sqrt{2 \pi}} \big) n \big) \leq 2\textrm{e}^{-
  \frac{(\lambda - \frac{\varepsilon}{\sqrt{2 \pi}})^2 n^2}{2
    \binom{n}{2} ( \frac2{\varepsilon \sqrt{n}} )^2}} \leq 2
\mathrm{e}^{-\varepsilon^2 (\lambda - \frac{\varepsilon}{\sqrt{2
      \pi}})^2 n / 4}~.
\]
Choosing $\lambda = c' n^{-\varepsilon'}$ and $\varepsilon = \tfrac12
\sqrt{2 \pi} c' n^{-\varepsilon'}$ concludes the proof.
\end{proof}

\begin{lemma}
\label{subsetofS}
Let $\tilde{S}_0 = \big\{ v \in S_0 : d_{S_0}(v) \geq \tfrac12 | S_0|
+ \eta \tfrac{\sqrt{| S_0|}}2 \big\}$. Then for every $0 <
\varepsilon_1 < \tfrac14$, whp($\mathrm{e}^{-\Theta(n^{1 - 4
    \varepsilon_1})}$) we have $\big| |\tilde{S}_0| - \gamma n \big|
\leq O(n^{1 - \varepsilon_1})$, where $\gamma = \alpha
\Phibar(\eta)$. Furthermore, for every $0 < \varepsilon_2 < \tfrac12$,
whp($\mathrm{e}^{-\Theta(k^{1 - 2 \varepsilon_2})}$) we have $\big|
|\tilde{S} \cap K| - \delta k \big| \leq O(k^{1 - \varepsilon_2})$,
where $\delta = \alpha\Phibar(\eta - c \sqrt{\alpha})$.
\end{lemma}
\begin{proof}
By \lemref{sizeofS}, whp($\mathrm{e}^{-\Theta(n^{1 -2 \varepsilon_1})}
+ \mathrm{e}^{-\Theta(k^{1 - 2 \varepsilon_2})}$) the size of $S_0$ is
$(1 + o(1)) \alpha n$ and the number of clique vertices in $S_0$ is
$(1 + o(1)) \alpha k$. The first part of the Lemma follows directly
from \lemref{numberabovethresholddependent} by setting $\varepsilon' =
\varepsilon_1$. For the second part of the Lemma, consider a clique
vertex $v \in S_0$. Having $d_{S_0}(v) \geq \tfrac12 \alpha n + \eta
\tfrac{\sqrt{\alpha n}}2$ is equivalent to having
\[
d_{S_0 \setminus
  K}(v) \geq \tfrac12 \alpha (n - k) + \tfrac12 (\eta - c
\sqrt{\alpha}) \sqrt{\tfrac{n}{n-k}} \sqrt{\alpha (n-k)}~.
\] 
Thus, setting $\varepsilon = \varepsilon_2$ in
\corref{numberabovethresholdindependent}, gives that
whp($\mathrm{e}^{-\Theta(k^{1 - 2 \varepsilon_2})}$), $\big|
|\tilde{S}_0 \cap K| - \delta' k \big| \leq O(k^{1 - \varepsilon_2})$,
where $\delta' = \alpha \Phibar \big( (\eta - c \sqrt{\alpha})
\sqrt{\tfrac{n}{n-k}} \big)$. The difference between $\delta$ and
$\delta'$ is of order $\tfrac1{\sqrt{n}}$, which means that the result
holds for $\big| | \tilde{S}_0 \cap K| - \delta k \big|$ as well.
\end{proof}

\begin{theorem}
Consider the variant of the algorithm, where $\tilde{V}_i$ is defined
by
\[
\tilde{V}_i = \big\{ v \in V(G_i) \setminus S_i: d_{\tilde{S}_i}(v)
\geq \tfrac12 |\tilde{S}_i| + \beta \tfrac{\sqrt{|\tilde{S}_i|}}2
\big\}
\]
with $\tilde{S}_i, \gamma$ as defined in \lemref{subsetofS}. If $c
\geq 1.261$ then there exist $\alpha, \beta, \eta$ for which running
the variant of the algorithm described above on a random graph in
$G(n, \tfrac12, c \sqrt{n})$ finds the hidden clique
whp($\mathrm{e}^{-\Theta(n^{\varepsilon_0})}$) for some $\varepsilon_0
= \varepsilon_0(c)$.
\end{theorem}
\begin{proof}
We follow the proof of \thmref{thetheorem}, with two differences. The
first is that we use \lemref{subsetofS} instead of \lemref{sizeofS},
which implies that instead of demanding $\varepsilon_1 < \tfrac12$ we
demand $\varepsilon_1 < \tfrac14$. The second is that in
\lemref{sizeofnextiteration} and everything that follows we use a
different definition for $\rho$. Since now the clique vertices' degree
boost is $c \tfrac{\delta}{\sqrt{\gamma}}$ times the standard
deviation, we define $\rho = (1 - \alpha) \Phibar \big( \beta -
\tfrac{c \delta}{\sqrt{\gamma}} \big)$. Next, for every $\alpha,
\beta, \eta$, we denote by $\tilde{c}(\alpha, \beta, \eta)$ the
minimal $c$ for which $\tfrac{\rho}{\sqrt{\tau}} \geq 1$. Denote the
infimum of $\tilde{c}(\alpha, \beta, \eta)$ by $c^*$. Numerical
calculations show that $c^*$ is close to $1.261$. The values of
$\alpha$, $\beta$ and $\eta$ for which which $\tilde{c}(\alpha, \beta,
\eta) = 1.261$ are $\alpha = 0.8$, $\beta = 2.3$ and $\eta = 1.2$. For
these values, we get $\tau \approx 0.0021448$ and $\rho \approx
0.046348$, and $\tfrac{\rho}{\sqrt{\tau}} \approx 1.0008$.
\end{proof}

\subsection{\texorpdfstring{Finding hidden dense graphs in $G(n,p)$}{}}
\label{hiddendensegraph}
Define the random graph model $G(n,p,k,q)$ for $0 < p < q < 1$. Given
a set of $n$ vertices, randomly choose a subset $K$ of $k$
vertices. For every pair of vertices $(u,v)$, the edge between them
exists with probability $p$ if at least one of the two vertices is in
$V \setminus K$, and with probability $q$ if they are both in $K$. The
model discussed in the previous sections is equivalent to $G \big( n,
\tfrac12, c \sqrt{n}, 1 \big)$.

Next, we define a generalization of the algorithm from the previous
section. This algorithm has the same three phases as before. In the
first phase, the definition of $\tilde{V}_i$ is
different. $\tilde{V}_i$ is defined as the set of vertices with at
least $p |S_i| + \beta \sqrt{p (1-p) |S_i|}$ neighbors in
$S_i$. Namely,
\[
\tilde{V}_i = \big\{ v \in V(G_i) \setminus S_i : d_S(v) \geq p |S_i|
+ \beta \sqrt{p (1-p) |S_i|} \big\}~.
\]
Define $\rho' = (1-\alpha) \Phibar \big( \beta - c \sqrt{\alpha}
\tfrac{q - p}{\sqrt{p(1-p)}}\big)$. In the second phase, after $t$
iterations, define $K'$ to be the set of $\rho'^t k$ largest degree
vertices in $G_t$, and let $\tilde{K}$ contain all the vertices in
$G_t$ that have at least $\tfrac12 (p + q)$ neighbors in $K'$. In the
third phase, let $K'$ be the set of vertices containing $\tilde{K}$
and all the vertices in $G$ that have at least $\tfrac12 (p + q) |
\tilde{K} |$ neighbors in $\tilde{K}$. Let $K^*$ be the set of all
vertices in $G$ that have at least $\tfrac12 (p + q)k$ neighbors in
$K'$. The algorithm returns $K^*$ as the candidate for the dense
graph.

\begin{theorem}
\label{theoremforgeneralp}
If $c > c_0 \tfrac{\sqrt{p(1-p)}}{q - p}$ then there exist $0 < \alpha
< 1$ and $\beta > 0$ for which given a graph $G \in G(n,p,c
\sqrt{n},q)$, the above algorithm finds the hidden dense graph
whp($\mathrm{e}^{-\Theta(n^{\varepsilon_0})}$) for $\varepsilon_0 =
\varepsilon_0(c)$.
\end{theorem}

To prove \thmref{theoremforgeneralp}, as in the hidden clique case, we
first prove the correctness of each of the phases of the algorithm,
and then bound the failure probability. To prove the correctness of
the first phase, we prove Lemmas \ref{nextiterationisrandompq} and
\ref{sizeofnextiterationpq}, which are analogous to Lemmas
\ref{nextiterationisrandom} and \ref{sizeofnextiteration}. To prove
the correctness of the second phase, we prove
\lemref{probcliqueverticesaremaximalpq} and
\corref{failureprobabilitytildeKpq}, which are analogous to
\lemref{probcliqueverticesaremaximal} and
\corref{failureprobabilitytildeK}. To prove the correctness of the
third phase we prove \lemref{findingKgivensubsetpq}. The failure
probability follows as in \lemref{totalfailureprobability} by noticing
that substituting $c \tfrac{\sqrt{p(1-p)}}{q-p}$ for $c$ in the
definition of $\rho'$ gives the exact definition of $\rho$.

\section{Discussion}
Our results bring up some interesting questions for future
research. For example, one of the advantages of the algorithm
presented here is a failure probability that is less than polynomially
small in the size of the input. Experimental results shown in
\cite{FeigeRon10} suggest that the failure probability of the
algorithm described there may also be $o(1)$. Whether the analysis can
be improved to prove this rigorously is an interesting open
question. One can also ask whether the analysis in
\cite{AlonKrivelevichSudakov98} can be improved to show failure
probability that is less than polynomially small.

Aside from the most interesting open question of whether there exists
an algorithm that finds hidden cliques for $k = o(\sqrt{n})$, one can
ask about ways to find hidden cliques of size $k = c \sqrt{n}$ as $c$
gets smaller. In \cite{AlonKrivelevichSudakov98}, Alon, Krivelevich
and Sudakov give a way to improve the constant for which their
algorithm works, at the expense of increasing the running time. This
technique can be used for any algorithm that finds hidden cliques, so
we describe it here. Pick a random vertex $v \in V$, and run the
algorithm only on the subgraph containing $v$ and its
neighborhood. $v$ is a clique vertex, then the parameters of the
algorithm have improved, since instead of having a graph with $n$
vertices and a hidden clique of size $c \sqrt{n}$ we now have a graph
with $\tfrac{n}2$ vertices and a hidden clique of size $c
\sqrt{n}$. The expected number of trials we need to do until we pick a
clique vertex is $O(\sqrt{n})$. This means that if we have an
algorithm that finds a hidden clique of size $c \sqrt{n}$, where $c
\geq c_0$, we can also find a hidden clique for $c \geq
\tfrac{c_0}{\sqrt{2}}$, while increasing the running time by a factor
of $\sqrt{n}$. If we wish to improve the constant even further, we can
pick $r$ random vertices and run the algorithm on the subgraph
containing them and their common neighborhood. This gives an algorithm
that works for constants smaller by up to a factor of $2^{r/2}$ than
the original constant, at the expense of increasing the running time
of the algorithm by a factor of $n^{r/2}$. 

We have described a sequence of algorithms whose running times
increase by factors of $\sqrt{n}$. It is not known whether the
constant can be decreased if we can only increase the running time by
a factor smaller than $\sqrt{n}$.

\begin{question}
Given an algorithm that runs in time $O(n^2)$ and finds hidden cliques
of size $c \sqrt{n}$ for any $c \geq c_0$, is there an algorithm that
runs in time $O(n^{2 + \varepsilon})$, where $\varepsilon < \tfrac12$,
and finds hidden cliques of size $c \sqrt{n}$ where $c < c_0$? How
small can $c$ be as a function of $\varepsilon$?
\end{question}

\bibliography{mybib} 
\bibliographystyle{plain}

\appendix
\section{Concentration inequalities}
Throughout the paper, we use the central limit theorem for binomial
random variables, and its rate of convergence that was independently
discovered by Berry in 1941 \cite{Berry41} and by Esseen in 1942
\cite{Esseen42}. For details, see, for example \cite[\S
  Sec. 3.4.4]{Durrett10}.
\begin{theorem}[Berry, Esseen]
\label{convergencetonormal}
Let $B(n,p)$ be a binomial random variable with parameters $n,
p$. Then for every $x \in \R$
\[
\big| \pr \big( \tfrac{B(n,p) - pn}{\sqrt{p(1-p) n}} \leq x \big) -
\Phi(x) \big| = O \big( \tfrac1{\sqrt{n}} \big)~.
\]
\end{theorem}

\begin{cor}
\label{expectednumberabovethreshold}
Let $B(n,p)$ be a binomial random variable. For any $a \in \R$, the
probability that $B(n,p)$ is greater than $pn + a \sqrt{p(1-p) n}$ is
bounded by
\[
\big| \pr \big( B(n,p) \geq pn + a \sqrt{p(1-p)n} \big) - \Phibar(a)
\big| \leq O \big( \tfrac1{\sqrt{n}} \big)~.
\]
\qed
\end{cor}

\begin{theorem}[Hoeffding's Inequality]
\label{hoeffding}
Let $S = X_1 + \cdots + X_n$ where the $X_i$'s are independent
Bernoulli random variables. Then for every $t > 0$
\[
\pr \left( \left|S - \E S \right| \geq t \right) \leq 2 \mathrm{e}^{-2
  t^2 / n}~.
\]
\end{theorem}

\begin{cor}
\label{numberabovethresholdindependent}
Let $A,B$ be two disjoint sets of vertices in $G \in G(n,p)$ with $|A|
= n_1$ and $|B| = n_2$ such that $n_1 \leq O \left( n_2 \right)$
. Given $a \in \R$, define the random variable
\[
X = \big| \big\{ v \in A : d_B(v) \geq pn_2 + a \sqrt{p(1-p)n_2}
\big\} \big|~.
\]
Then for every $c' > 0$ and $0 < \varepsilon < \tfrac12$ it holds that
\[
\pr \big( \big| X - \Phibar(a)n_1 \big| \geq c' n_1^{1 - \varepsilon}
\big) \leq \mathrm{e}^{-c' n^{1 - 2 \varepsilon} / 2}~.
\]
\end{cor}
\begin{proof}
From \corref{expectednumberabovethreshold} we know that $\left|
\Phibar(a)n_1 - \E X \right| \leq c \tfrac{n_1}{\sqrt{n_2}}$ for some
constant $c > 0$. Therefore, by \thmref{hoeffding}, for any constant
$c'>0$,
\begin{eqnarray*}
\pr \big( \big| X - \Phibar(a) n_1 \big| \geq c' n_1^{1 - \varepsilon}
\big) & \leq & \pr \big( \big| X - \E X \big| \geq c' n_1^{1 -
  \varepsilon} - c \tfrac{n_1}{\sqrt{n_2}} \big) \\ & \leq &
\mathrm{e}^{-(c' n_1^{1 - \varepsilon} - c n_1/\sqrt{n_2})^2/n_1} \leq
\mathrm{e}^{-\frac12 c'n_1^{1 - 2 \varepsilon}}
\end{eqnarray*}
where the last inequality holds because $\tfrac{n_1}{\sqrt{n_2}} \leq
O \big( \sqrt{n_1} \big) = o ( n_1^{1 - \varepsilon} )$.
\end{proof}

\section{\texorpdfstring{The $G(n,p,k,q)$ case}{}}
\begin{lemma}[analogous to \lemref{nextiterationisrandom}]
\label{nextiterationisrandompq}
For every $i \geq 0$, the graph $G_i$ defined the $i$'th iteration of
the algorithm is a copy of $G(\tilde{n}_i, p, \tilde{k}_i,q)$.
\end{lemma}
\begin{proof}
The proof is identical to the proof of \lemref{nextiterationisrandom}.
\end{proof}

\begin{lemma}[analogous to \lemref{sizeofnextiteration}]
\label{sizeofnextiterationpq}
For every $0 < \varepsilon_1, \varepsilon_2 < \tfrac12$, the set
$\tilde{V}$ satisfies $ \big| | \tilde{V} | - \tau n \big| \leq O(n^{1
  - \varepsilon_1})$ and $ \big| | \tilde{V} \cap K | - \rho' k \big|
\leq O(k^{1 - \varepsilon_2})$ whp($\mathrm{e}^{-\Theta(n^{1 - 2
    \varepsilon_1})} + \mathrm{e}^{-\Theta(k^{1 - 2
    \varepsilon_2})}$).
\end{lemma}
\begin{proof}
Follows from \corref{numberabovethresholdindependent} the same way as
in the proof of \lemref{sizeofnextiteration}.
\end{proof}

\begin{lemma}[analogous to \lemref{probcliqueverticesaremaximal}]
\label{probcliqueverticesaremaximalpq}
Let $G \in G(n,p,k,q)$ where $k \geq c_0 \sqrt{n \log n}$. Denote the
hidden dense graph by $K$ and the set of $k$ largest degree vertices
by $M$. Then
\[
\pr \big( \big| K \setminus M \big| > 0 \big) \leq \mathrm{e}^{-(q -
  p) k^2 / 2 n - \log n - O(1)}~.
\]
\end{lemma}
\begin{proof}
Define $x = \tfrac12 (q - p) k$. Then by \thmref{hoeffding}
\[
\pr \big( \exists v \not\in K : d(v) \geq p n + x \big) \leq n \pr
\big( B \big( n,p \big) \geq p n + x \big) \leq n \pr \big( \big| B
\big(n, p \big) - p n \big| \geq x \big) \leq 2 n \mathrm{e}^{-(q -
  p)^2 k^2 / 2n}~.
\]
On the other hand,
\begin{eqnarray*}
\pr \big( \exists v \in K : d(v) < p n + x \big) & \leq & k \pr \big(
B \big( n - k,p \big) + B \big( k, q \big) - p (n-k) - q k < x - (q -
p) k \big) \\ & \leq & k \pr \big( \big| B \big(n-k, p \big) + B
\big(k, q \big) - p (n-k) - q k \big| \geq x \big) \\ & \leq & 2 k
\mathrm{e}^{-(q - p)^2 k^2 / 2n}~.
\end{eqnarray*}
Therefore, the probability that there exist a vertex $v \not\in K$ and
a vertex $u \in K$ such that $d(u) < d(v)$ is bounded by $2 (n+k)
\mathrm{e}^{-(q - p)^2 k^2/2n}$.
\end{proof}

\begin{cor}[analogous to \corref{failureprobabilitytildeK}]
\label{failureprobabilitytildeKpq}
If the algorithm does $t$ iterations before finding $\tilde{K}$ and
succeeds in every iteration, then whp($\mathrm{e}^{-\Theta ((
  \frac{\rho^2}{\tau})^t)}$), $\tilde{K}$ is a subset of the original
hidden dense graph.
\end{cor}
\begin{proof}
The proof is analogous to the proof of
\corref{failureprobabilitytildeK}, by noticing that
whp($\mathrm{e}^{-\Theta (\frac{\rho^{2t}k^2}{\tau^t n})}$), every
hidden dense graph vertex in $G_t$ has at least $\big( q - \tfrac{q -
  p}4 \big) k_t - o(k_t)$ neighbors in $K'$ and every non-hidden dense
graph vertex in $G_t$ has at most $\big( p + \tfrac{q - p}4 \big) k_t
+ o(k_t)$ neighbors in $K'$.
\end{proof}

\begin{lemma}
\label{findingKgivensubsetpq}
We are given a random graph $G \in G(n,p,k,q)$, and also a subset of
the hidden dense graph $\tilde{K}$ of size $s$. Denote the hidden
dense graph in $G$ by $K$. Suppose that either
\begin{enumerate}[(a)]
\item $k = O(\log n \log \log n)$ and $s \geq \big( \tfrac2{(q - p)^2}
  + \varepsilon \big) \ln n$ for some $\varepsilon > 0$, or
\item $k \geq \omega (\log n \log \log n)$ and $s \geq \tfrac2{(q -
  p)^2} \ln n + 1$.
\end{enumerate}
Let $K'$ denote the set of vertices containing $\tilde{K}$ and all the
vertices in $G$ that have at least $\tfrac12 (p + q)s$ neighbors in
$\tilde{K}$. Define $K^*$ to be the set of vertices of $G$ that have
at least $\tfrac12 (p + q)k$ neighbors in $K'$. Then for every $0 <
\varepsilon_3 < \tfrac12$, whp($\mathrm{e}^{-\Theta(s \log k + \log
  n)} + \mathrm{e}^{-\Theta(k^{1 - 2 \varepsilon_3})}$), $K^* = K$.
\end{lemma}
\begin{proof}
Consider an arbitrary subset $S$ of $K$ of size $s$. By
\thmref{hoeffding}, the probability that a specific vertex $v \not\in
K$ has more than $\tfrac12 (p + q)s$ neighbors in $S$ is bounded by
$\mathrm{e}^{-(q - p)^2 s / 2}$. The probability that a specific
vertex $v \in K$ has less than $\tfrac12 (p + q)s$ neighbors in $S$ is
bounded by the same expression. Therefore, the probability of having
at least $l_0$ ``bad'' vertices (where ``bad'' is defined by either a
vertex of $K$ that is not in $K'$ or a vertex not in $K$ that is in
$K'$) is bounded by $\sum_{l=l_0}^n n^l \mathrm{e}^{-(q-p)^2 s l /
  2}$. Taking union bound over all subsets of size $s$ of $K$ gives
that the probability that there exists a subset with at least $l_0$
bad vertices is bounded by
\[
k^s \sum_{l=l_0}^n \mathrm{e}^{l (\ln n - (q - p)^2 s / 2)} \leq n
\mathrm{e}^{s \ln k - l_0 ((q - p)^2 s / 2 - \ln n)} = \mathrm{e}^{\ln
  n + s \ln k - l_0 ((q-p)^2 s / 2 - \ln n)}~.
\]
If we take $l_0 = \tfrac{2 (\ln n + s \ln k)}{(q-p)^2 s / 2 - \ln n}$
this probability is $\mathrm{e}^{-\ln n - s \ln k}$. Therefore,
whp($\mathrm{e}^{-\ln n - s \ln k}$) there are at most $l_0$ bad
vertices in $K'$. Specifically, this implies that $K'$ contains at
least $k - l_0$ vertices from $K$ and at most $l_0$ vertices not from
$K$, and that $|K'| \leq k + l_0$. By \thmref{hoeffding} and the union
bound, the probability that there exists a vertex $v \in K$ with less
than $q k - k^{1 - \varepsilon_3} $ neighbors in $K$ is bounded by
$\mathrm{e}^{-\Theta(k^{1 - 2 \varepsilon_3})}$, and so is the
probability that there exists a vertex $v \not\in K$ with more than
$pk + k^{1 - \varepsilon_3}$ neighbors in $K$. Therefore,
whp($\mathrm{e}^{-\Theta(k^{1 - 2 \varepsilon_3})}$) the number of
neighbors every $v \in K$ has in $K'$ is at least $qk - k^{1 -
  \varepsilon_3} - l_0$, and the number of neighbors every $v \not\in
K$ has in $K'$ is at most $pk + k^{1 - \varepsilon_3} + l_0$. Thus, if
$s$ and $k$ are such that $l_0 = o(k)$ then whp($\mathrm{e}^{-\ln n -
  s \ln k} + \mathrm{e}^{-\Theta(k^{1 - 2 \varepsilon_3})}$) $K^* =
K$.
\end{proof}

\end{document}